\documentclass{amsart}
\usepackage{amsmath,amssymb,amsthm}
\usepackage{comment,cleveref,enumerate}
\usepackage[all]{xy}

\crefname{equation}{equation}{equations}
\crefname{section}{Section}{Sections}
\crefname{chapter}{Chapter}{Chapters}

\newtheorem{lem}{Lemma}[section]
\newtheorem{thm}[lem]{Theorem}

\newtheorem{cor}[lem]{Corollary}
\newtheorem{prop}[lem]{Proposition}

\theoremstyle{definition}
\newtheorem{dfn}[lem]{Definition}

\newtheorem{question}[lem]{Question}
\newtheorem{notation}[lem]{Notation}
\newtheorem{ex}[lem]{Example}
\newtheorem{rmk}[lem]{Remark}

\newcommand{\ZZ}{\mathbb{Z}}

\newcommand{\NN}{\mathbb{N}}
\newcommand{\CC}{\mathbb{C}}
\newcommand{\FF}{\mathbb{F}}

\newcommand{\PP}{\mathbb{P}}

\newcommand{\Spec}{\operatorname{Spec}}

\newcommand{\Aut}{\operatorname{Aut}}   % Automorphism group

\newcommand{\PGL}{\operatorname{PGL}}
\newcommand{\PSL}{\operatorname{PSL}}
\newcommand{\cha}{\operatorname{char}}

\newcommand{\pie}{\pi_1^{\acute{e}t}}
\newcommand{\e}{\'{e}}

\newcommand{\benum}{\begin{enumerate}[(1)]}
\newcommand{\eenum}{\end{enumerate}}

\numberwithin{equation}{section}
\usepackage{comment,cleveref,enumerate}

\begin{document}

\title{Galois branched covers with fixed ramification locus}

\author{Ryan Eberhart}
\address{Department of Mathematics, The Pennsylvania State university, University Park, PA 16802}
%\curraddr{}
%\email{}
%\thanks{}


\date{}

\dedicatory{}

\begin{abstract}
We examine conditions under which there exists a non-constant family of Galois branched covers of curves over an algebraically closed field $k$ of fixed degree and fixed ramification locus, under a notion of equivalence derived from considering linear series on a fixed smooth proper source curve $X$. We show such a family exists precisely when the following conditions are satisfied: $\cha(k)=p>0$, $X$ is isomorphic to $\PP^1_k$, there is a unique ramification point, and the Galois group is $(\ZZ/p\ZZ)^m$ for some integer $m>0$.
\end{abstract}

\maketitle

\section{Introduction}\label{introsect}
This paper concerns the problem of determining the existence of a non-constant family of branched covers of curves, after prescribing some attributes of the covers. Let $k$ be an algebraically closed field. To avoid awkward phrasing, in this section always assume that families are over a connected base scheme. When considering covers of a fixed target curve $Y$, one generally considers two covers $f_i:X_i\rightarrow Y$ equivalent if there is a commutative diagram:

$$\xymatrix{
X_1\ar[rr]^\cong \ar[dr]_{f_1} & & X_2\ar[dl]^{f_2} \\
& Y &
}$$

In this situation, if one specifies that the branch locus be contained in a finite set of closed points $S\subseteq Y$, properties of the \e tale fundamental group $\pie(Y\setminus S)$ can be used to establish results concerning families of covers with branch locus contained in S. For instance, if $\cha(k)=0$, then $\pie(Y\setminus S)$ is finitely generated by Riemann's Existence Theorem. This implies that there is no non-constant family of branched covers with target $Y$ of fixed degree and branch locus contained in $S$. Furthermore, by \cite[Expos\e\ XIII Corollaire 2.12]{sga1} the tame \e tale fundamental group of $Y\setminus S$ is finitely generated. Thus, if $\cha(k)>0$, there is no non-constant family of tamely ramified branched covers with target $Y$ of fixed degree and branch locus contained in $S$.

If instead we begin from the perspective of linear series, or finite index subfields of the function field of a fixed curve, we obtain a different notion of equivalence, as we now explain. Fix a smooth proper curve $X$ over $k$. A \textit{linear series} on $X$ is a linear subspace $V\subseteq H^0(X,\mathcal{L})$ where $\mathcal{L}$ is a line bundle on $X$. Two sections $s_1$ and $s_2$ of a line bundle that have no common zeroes determine a map $X\rightarrow\PP^1_k$, taking $P\in X$ to $(s_1(P):s_2(P))$. Conversely, a map to $\PP^1_k$ yields two sections of a line bundle with no common zeroes, by pulling back generators for the sheaf $\mathcal{O}(1)$. Hence, in studying a linear series $V$, it is natural to investigate planes inside $V$. One can easily see that choosing a different basis for the plane corresponds to post-composing the map determined by $s_1$ and $s_2$ with a fractional linear transformation (i.e.\ an automorphism of the target space, $\PP^1_k$). Therefore, planes inside a linear series 
correspond to maps 
$X\rightarrow\PP^1$ up to post-composition with fractional linear transformations.

Motivated by this discussion, we will henceforth consider two branched covers $f_i:X \to Y_i$ to be \textit{equivalent} if there is a commutative diagram:

$$\xymatrix{
& X\ar[ld]_{f_1}\ar[rd]^{f_2} & \\
Y_1\ar[rr]^\cong & & Y_2
}$$
This notion of equivalence is referred to as the \textit{linear series perspective} in \cite{ossper}. \cref{subfields} below explains the correspondence between branched covers with source $X$ under this equivalence and subfields of $\kappa(X)$.

In contrast with the situation of covers of a fixed base, we are fixing the source and allowing the target to vary. Therefore instead of fixing the branch locus, we consider covers with a fixed ramification locus on the source. In analogy with the situation of covers of a fixed target curve, one is led to consider the following:

\begin{question}\label{quest}
Let $X$ be a smooth proper curve over an algebraically closed field $k$ and $S$ a finite set of closed points of $X$. Under what conditions does there exist a non-constant family of maps with source $X$ of fixed degree and ramification locus $S$?
\end{question}

Based upon the results concerning covers of a fixed target curve, one would expect no such tame families to exist. However, as noted by Brian Osserman, the following example illustrates that this is not the case:

\begin{ex}\label{tameex} (\cite[Example 5.6]{ossp})
Consider the family of maps $\PP^1_k\rightarrow \PP^1_k$ with $k$ a field of characteristic $p>2$ given by $y=x^{p+2}+tx^p+x$ with parameter $t$. For every value of $t$, the map is tamely ramified at $\infty$ and the $(p+1)^{st}$ roots of $-1/2$ and is \e tale elsewhere. When fixing the source $\PP^1_k$, no distinct values of $t$ produce maps which are equivalent after an isomorphism of the target $\PP^1_k$ (the linear series perspective sense of equivalence which we are considering).
\end{ex}

To summarize, in the notion of equivalence arising from considering branched covers of a fixed target curve, no non-constant families of tame covers with fixed degree and branch locus exist. However, when using the equivalence arising from considering linear series on a fixed source curve, non-constant families of tame covers of fixed degree and ramification locus can exist, such as the non-Galois family from \cref{tameex}. When such families exist is an open problem. In the case of Galois covers, we show that non-constant families must be wild in \cref{notamegalcor}. Moreover, \cref{maingalthm} provides the precise conditions under which such a family can exist. The restrictiveness of these conditions shows that Galois families are rarer under this sense of equivalence than the usual one.

\medskip

{\it Structure of the paper:} In \cref{familysection}, we define a family for branched covers under the above notion of equivalence. The main result of this section is \cref{isotrivthm}, which allows us to translate statements concerning the number of equivalence classes to statements about the existence of families. \cref{nongalsect} provides a survey of results in the situation where the covers are not assumed to be Galois.
We restrict our attention to covers which are Galois in \cref{galsect}. Equivalence classes of such covers with source $X$ are shown to correspond to certain group actions on $X$. By examining such group actions, we prove the main result of this paper, \cref{maingalthm}. This theorem completely answers \cref{quest} in the case where the covers are required to be Galois.

\medskip

{\it Terminology and Notation:} For any integral scheme $X$, $\kappa(X)$ will denote the function field of $X$. In this paper we consider smooth connected curves over an algebraically closed field $k$. A morphism $f:X\rightarrow Y$ of such curves is a \textit{branched cover} if it is finite and generically \e tale. We will occasionally use the term \textit{map} as shorthand for a branched cover. A branched cover $f:X\rightarrow Y$ is \textit{Galois} if $\Aut(X/Y)$, the group of $k$-automorphisms of $X$ commuting with $f$, has cardinality equal to the degree of $f$. Equivalently, $f:X\rightarrow Y$ is \textit{Galois} if the induced field extension $\kappa(X)/\kappa(Y)$ is Galois. A Galois cover $f:X\rightarrow Y$ is \textit{$G$-Galois} if $\Aut(X/Y)$ is isomorphic to $G$.

\medskip

The content of this paper is adapted from a portion of the author's PhD thesis at the University of Pennsylvania, under the direction of David Harbater.

%%%%%%%%%%%%%%%%%%%
%%%%%%%%%%%%%%%%%%%

\section{Families}\label{familysection}
In formulating the definition of a family, the essential point we wish to capture is that in our notion of equivalence the source $X$ is fixed but the target is allowed to vary.

\begin{dfn}\label{famdfn}
Let $k$ be an algebraically closed field, $X$ a smooth proper curve over $k$, and $S$ a $k$-scheme. A \textit{family of degree d maps with source X over $S$} is an $S$-morphism $f:X_S\rightarrow \mathcal{Y}$ where $\mathcal{Y}$ is a flat $S$-scheme and the geometric fibers of $f$ are degree $d$ morphisms between smooth proper curves. A family is \textit{constant} if the fibers over all $k$-points of $S$ are equivalent.
\end{dfn}

Previous results in the literature, as well as our method of proof for \cref{maingalthm}, show that there are finitely many maps satisfying a set of conditions. Therefore we need a result to translate between statements concerning the number of equivalence classes of maps and statements about families:

\begin{thm}\label{isotrivthm}
Let $k$ be an algebraically closed field, $X$ a smooth proper curve over $k$, $S$ a connected $k$-variety, $\mathcal{Y}$ an $S$-scheme, and $f:X_S\rightarrow \mathcal{Y}$ a family of degree d maps over $S$. If the maps over $k$-points of $S$ lie in finitely many equivalence classes, then $f$ is constant.
\end{thm}

First, we will prove \cref{isotrivthm} for $\mathcal{Y}=Y_S$ and then reduce the general case to this case. In this restricted case, we will make use of the following:

\begin{lem}(\cite[Lemma 1.8]{borel} - Closed Orbit Lemma)\label{orbitlem}
Let $k$ be an algebraically closed field, $G$ a smooth $k$-group of finite type over $k$, $X$ a scheme of finite type over $k$, and $\alpha:G\times X\rightarrow X$ a group action. For $x\in X(k)$, let $\alpha_x:G\rightarrow X$ be the orbit map (i.e.\ $\alpha(g)=g\cdot x$). Then the set theoretic image of $\alpha_x$ is locally closed, and with the reduced induced scheme structure it is smooth. Moreover, the orbits of minimal dimension are closed.
\end{lem}

Equipped with \cref{orbitlem}, we can provide a proof of \cref{isotrivthm} in the $\mathcal{Y}=Y_S$ case:

\begin{prop}\label{isotrivprop}
Let $k$ be an algebraically closed field and let $X$ and $Y$ be smooth proper curves over $k$. Let $S$ be a connected $k$-variety, and $f:X_S\rightarrow Y_S$ a family of degree $d$ maps over $S$. If the maps over $k$-points of $S$ lie in finitely many equivalence classes, then $f$ is constant.
\end{prop}
\begin{proof}

By \cite[Theorem A.1]{ossp}, the functor $\mathcal{H}_d(X,Y)$ sending a $k$-scheme $S$ to $\{S$-morphisms $X_S\rightarrow Y_S\mbox{ of degree }d\}$ is representable by a quasi-projective scheme $H_d(X,Y)$. Hence, $f:X_S\rightarrow Y_S$ corresponds to a morphism $\varphi_f:S\rightarrow H_d(X,Y)$.

By \cite[Theorem A.6]{ossp} there is a free action of the group scheme $\Aut(Y)$ on $H_d(X,Y)$ where $\varphi \in \Aut(Y)$ acts on the point corresponding to a map $f:X\rightarrow Y$ by sending it to the point corresponding to $\varphi\circ f:X\rightarrow Y$. The freeness of the $\Aut(Y)$-action implies that each orbit has the same dimension. By \cref{orbitlem}, each orbit is then closed. Since the maps over $k$-points of $S$ lie in finitely many equivalence classes, $\varphi_f(S)$ is contained in finitely many $\Aut(Y)$-orbits. Since $S$ is connected, $\varphi_f(S)$ is connected as well, and must then be contained in a single $\Aut(Y)$-orbit. In other words, all the maps over $k$-points of $S$ lie in a single equivalence class, so the family is constant.
\end{proof}

Using the moduli stack of genus $g$ curves, we can reduce the general case to the situation of \cref{isotrivprop}:

\begin{proof}[Proof of \cref{isotrivthm}]
Let $\pi:\mathcal{Y}\rightarrow S$ be the $S$-structure morphism. Since $X$ is a proper curve and $f$ is finite, $\pi^{-1}(s)$ is also a proper curve for every $k$-point $s\in S$. Since $\pi$ is flat, the arithmetic genus of the fibers is constant and by \cref{famdfn} the fibers are smooth. Therefore $\pi$ induces a map $\varphi_\pi:S\rightarrow M_g$ to the coarse moduli space of genus $g$ curves over $k$. Since $S$ is connected and $\varphi_\pi(S)$ contains finitely many $k$-points by hypothesis, $\varphi_\pi(S)$ contains a single $k$-point. In other words, every $\pi^{-1}(s)$ is isomorphic. Let $Y$ denote this curve.

We claim that after some \e tale pullback $\mathcal{Y}$ becomes trivial. That is, there exists an \e tale covering $\psi:S'\rightarrow S$ such that the following commutes
$$\xymatrix{
\mathcal{Y}\times_S S'\ar[rr]^\cong \ar[dr] & & Y_{S'}\ar[dl] \\
& S' &
}$$
To prove this claim, let $\mathcal{M}_g$ denote the moduli stack of genus $g$ curves over $k$.
Since $\mathcal{Y}$ has constant fiber corresponding to a $k$-point of $M_g$ and $S$ is a variety, the maps it induces from $S$ to $\mathcal{M}_g$ and $M_g$ factor as in the commutative diagram

$$\xymatrix{
S\ar[r]\ar[d] & \mathcal{M}_g\ar[d] \\
\Spec(k)\ar[r] & M_g
}$$
which induces a map $S\rightarrow \mathcal{M}_g\times_{M_g} \Spec(k)$. The trivial family $Y\times_{\Spec(k)}S$ also induces a morphism from $S$ to $\mathcal{M}_g\times_{M_g} \Spec(k)$. Since $\mathcal{M}_g\times_{M_g} \Spec(k)\rightarrow\Spec(k)$ is a gerbe, every two objects over $S$ are locally isomorphic in the \e tale topology. That is, there is an \e tale covering $\psi: S'\rightarrow S$ over which there is an isomorphism $\xi:\mathcal{Y}\times_S S'\rightarrow Y_S'$. Hence, after an \e tale base change we have a family $\xi\circ(f\times \psi):X_{S'}\rightarrow Y_{S'}$, which concludes the proof of the claim.

Denote $\xi\circ (f\times \psi)$ by $g$. Since $\psi:S'\rightarrow S$ is \e tale, if $\psi(s')=s$ then the fiber of $g$ above $s'$ is equivalent to the fiber of $f$ above $s$. Thus, since there are finitely many equivalence classes of fibers of $f$, the same is true for $g$. Therefore by \cref{isotrivprop}, each connected component of $S'$ has only one equivalence class of fiber lying over its $k$-points. Since \e tale morphisms are open by \cite[Proposition 2.4.6]{ega4} and $S$ is connected, it is straightforward to verify that this implies the same holds for $S$.
\end{proof}

\begin{rmk}
Whenever we say having a finite number of equivalence classes of fibers over a connected base implies that there are no non-constant families, it will be by \cref{isotrivthm}. Furthermore, when stating that there are no non-constant families it will be implicit that we are restricting to families over a connected base.
\end{rmk}

%%%%%%%%%%%%%%%%%%%
%%%%%%%%%%%%%%%%%%%

\section{Summary of results in the non-Galois case}\label{nongalsect}
Various results from the literature presented in this section state that there are finitely many equivalence classes of maps satisfying certain properties. Using \cref{isotrivthm} we can translate these results into statements about the non-existence of a non-constant family.

\subsection{Source of genus 0}
When $k$ is the complex numbers, \cref{quest} can be answered using a theorem of Eisenbud and Harris when $X$ is the projective line:

\begin{thm}(\cite[Theorem 2.3]{eh})\label{eh}
Let $S$ be a finite set of closed points of $\PP^1_\CC$ and $d\in\NN$. The number of equivalence classes of degree $d$ maps with ramification locus $S$ is finite.
\end{thm}

This is proven within the framework of Schubert calculus. More specifically, by the discussion in \cref{introsect}, equivalence classes of degree $d$ maps correspond to  points on the Grassmannian $Gr(2,H^0(\mathcal{O}(d))$.  In \cite{eh}, each ramification condition is translated as the plane lying on a certain Schubert variety. The intersection of these Schubert varieties is shown to be transverse and having expected dimension zero, thus establishing the finiteness result.

The following result applies the methods of \cref{eh} in positive characteristic. The restriction that the degree of the map is smaller than the characteristic of the field is necessary because the existence of inseparable maps in $Gr(2,H^0(\mathcal{O}(d))$ may cause the intersection of the Schubert varieties corresponding to imposing ramification conditions not to be transverse.

\begin{thm}(\cite[Corollary 3.2]{ossp}) \label{osschar}
Let $k$ be an algebraically closed field of characteristic $p>0$, $d$ a positive integer less than $p$, and $S$ a finite set of closed points of $\PP^1_k$. The number of equivalence classes of degree $d$ maps with ramification locus $S$ is finite.
\end{thm}

In light of \cref{osschar}, any non-constant family of maps with source $\PP^1_k$ of fixed degree and ramification locus must have degree at least as large as the characteristic of the field. Indeed, this is the case in \cref{tameex}.

\subsection{Target of higher genus}
In \cite{fra}, de Franchis proved a strong statement concerning the number of maps from a fixed curve $X$ to all curves of genus greater than one. In modern language it can be formulated as follows:

\begin{thm}(\cite[Theorem of de Franchis]{kani})\label{defrthm}
Let $X$ be a smooth proper curve over a field and $M=\{f:X\rightarrow Y|$ $f$ is non-constant and the genus of $Y$ is greater than one$\}$. Consider $f_i:X\rightarrow Y_i$, \textit{equivalent} if there is an isomorphism $g:Y_1\rightarrow Y_2$ such that $g\circ f_1=f_2$. Then the number of equivalence classes in $M$ is finite.
\end{thm}

Since the target curve has genus at least two, $\Aut(Y)$ is finite. Therefore one can rephrase \cref{defrthm} as stating that there are finitely many non-constant maps $f:X\rightarrow Y$, where the target curve $Y$ varies over a set of representatives for the isomorphism classes of curves of genus at least two.

Using \cref{subfields} below, we can translate \cite[Theorem 4]{kani} which concerns subfields of $\kappa(X)$ into a statement concerning equivalence classes of branched covers. This reformulation, along with the fact that there are finitely many equivalence classes of isogenies from a fixed elliptic curve $E$ of bounded degree, yields the following:

\begin{thm}\label{tog1}
Let $X$ be a smooth proper curve over an algebraically closed field $k$ and $d\in\NN$. The number of equivalence classes of degree $d$ maps from $X$ to genus one curves is finite.
\end{thm}

%%%%%%%%%%%%%%%%%%%
%%%%%%%%%%%%%%%%%%%

\section{The case of Galois covers}\label{galsect}
For the remainder of this paper, we shall restrict our attention to Galois covers. The culminating result of this section is \cref{maingalthm}, which answers the following:

\begin{question}\label{galquest}
Let $X$ be a smooth proper curve over an algebraically closed field $k$, $S$ a finite set of closed points of $X$, and $G$ a finite group. Under what conditions does there exist a non-constant family of $G$-Galois maps with source $X$, of fixed degree and ramification locus $S$?
\end{question}

The first step of the proof of \cref{maingalthm} is establishing a correspondence between equivalence classes of $G$-Galois branched covers with source $X$ and $G$-actions on $X$. This is accomplished in \cref{actoncurves}. In that section it is also shown that the stabilized points of the $G$-action (see \cref{stabdfn} below) are exactly the ramification points of the induced $G$-Galois branched cover.

With this, fixing the ramification locus $S$ as above, we are reduced to analyzing $G$-actions on $X$ with stabilized locus $S$. After this reduction, \cref{galquest} can be answered almost immediately for the case when $X$ has genus greater than one. Two essentially disjoint proofs are given for the genus zero and genus one cases.

\subsection{Correspondence between group actions and classes of covers}\label{actoncurves}
We first establish a correspondence between equivalence classes of branched covers with source $X$ and subfields of $\kappa(X)$, which applies even without imposing the Galois condition:

\begin{lem}\label{subfields}
Let $X$ be a smooth proper curve over an algebraically closed field $k$. Finite index subfields of $\kappa(X)$
are in correspondence with equivalence classes of branched covers with source $X$. In this correspondence, the degree of the subfield is the degree of the branched cover.
\end{lem}

\begin{proof}
Suppose $Y_1$ and $Y_2$ are also smooth proper curves and the branched covers $f_i:X\rightarrow Y_i$ are equivalent. That is, there is a commutative diagram:

$$\xymatrix{
& X\ar[ld]_{f_1}\ar[rd]^{f_2} & \\
Y_1\ar[rr]^\cong & & Y_2
}$$
By the equivalence of categories between smooth proper curves over $k$ and fields of transcendence degree one over $k$, this yields a commutative diagram of fields:

$$\xymatrix{
& \kappa(X) & \\
\kappa(Y_1)\ar[ur]^{i_1} & & \kappa(Y_2)\ar[ll]_\cong\ar[ul]_{i_2}
}$$
Therefore $i_1(\kappa(Y_1))=i_2(\kappa(Y_2))$ as subfields of $\kappa(X)$. The index of the subfield is equal to the degree of $f_i$. This shows that the mapping to subfields is well-defined on equivalence classes of covers.

Given a finite index subfield $L\subseteq \kappa(X)$, $L$ is also a field of transcendence degree one over $k$. Let $Y$ be a smooth proper curve over $k$ with function field $L$, which is unique up to isomorphism, and $X\rightarrow Y$ the branched cover corresponding to $\kappa(X)/L$. The uniqueness up to isomorphism of $Y$ is precisely the notion of equivalence we are considering. Therefore, this subfield gives a well-defined equivalence class of branched cover. These two operations are clearly inverse to each other, establishing the desired correspondence.
\end{proof}

To begin speaking of $G$-Galois covers, we introduce the definition of a group action on $X$:

\begin{dfn}
Let $X$ be a smooth proper curve over an algebraically closed field $k$ and $G$ a finite group. A \textit{group action} on $X$ is a choice of a finite subgroup of $\Aut(X)$. A \textit{$G$-action} on $X$ is a choice of a subgroup $H\subseteq \Aut(X)$ which is isomorphic to $G$. The choice of such an $H$ corresponds to an equivalence class of injective group homomorphisms $i:G\rightarrow \Aut(X)$, where two injections are considered equivalent if their images are the same.
\end{dfn}

\begin{rmk}
If one wished to consider injections of $G$ into $\Aut(X)$ instead of equivalences classes of such injections, to preserve the correspondence with $G$-Galois branched covers in \cref{subgroupstocovers} below, one would need to include in the definition of a $G$-Galois cover an isomorphism of $\Aut(X/Y)$ with $G$. Such a definition is used for instance in \cite{ch}.
\end{rmk}

To answer \cref{galquest}, we must also reformulate ramification of $G$-Galois branched covers into a statement about $G$-actions on $X$. To accomplish this, we define the following:

\begin{dfn}\label{stabdfn}
A closed point $P$ on $X$ is a \textit{stabilized point} of a $G$-action if its stabilizer under the $G$-action is non-trivial. The set of stabilized points of a $G$-action is its \textit{stabilized locus}. A $G$-action is \textit{stabilized point free} if its stabilized locus is empty.
\end{dfn}

Using \cref{subfields} and the above definitions, we establish the following correspondence between group actions and equivalence classes of covers:

\begin{thm}\label{subgroupstocovers}
Fix a smooth proper curve $X$ over an algebraically closed field $k$ and a finite group $G$. There is a correspondence between equivalence classes of $G$-Galois branched covers with source $X$ and $G$-actions on $X$. Under this correspondence, the ramification points of the cover are the stabilized points of the action.
\end{thm}

\begin{proof}
We first establish the correspondence. Fix a $G$-action on $X$. This action induces a $G$-action on $\kappa(X)$. By taking the $G$-invariants, we obtain a $G$-Galois field extension $\kappa(X)/ \kappa(X)^G$. By \cref{subfields}, this field extension corresponds to an equivalence class of $G$-Galois branched covers.

If $H,H'\subseteq\Aut(X)$ are distinct $G$-actions, by Galois correspondence the invariant subfields of the actions on $\kappa(X)$ will be distinct. Therefore by \cref{subfields}, these actions correspond to distinct equivalence classes of branched covers. These two operations are inverse to each other, establishing the desired correspondence.

Finally, since $X$ is defined over an algebraically closed field, extensions of residue fields are trivial. Thus, a point of $X$ being a stabilized point is equivalent to the point being a ramification point.
\end{proof}

\subsection{Statement of the Theorem}\label{maingalsect}
By \cref{actoncurves}, we reduce to studying the $G$-actions on $X$ which have as stabilized points a prescribed finite set of points. We will pursue this program to prove the following:

\begin{thm}\label{maingalthm}
Let $X$ be a smooth proper curve over an algebraically closed field $k$, $S$ a finite set of closed points of $X$, and $G$ a finite group. There exist infinitely many equivalence classes of $G$-Galois covers with ramification locus $S$ if and only if all the following hold:

\benum
\item $\cha(k)=p>0$

\item $X\cong \mathbb{P}^1_k$

\item $G\cong (\mathbb{Z}/p\ZZ)^m$, where $m\in\NN$

\item $|S|=1$.
\eenum

When these conditions hold, equivalence classes of such covers ramified at $S$ are in correspondence with rank $m$ additive subgroups of $k$, as described in \cref{wildp1}. Furthermore, when there are finitely many equivalence classes of $G$-Galois covers, there are no non-constant families by \cref{isotrivthm}.
\end{thm}

\begin{rmk}\label{maingalrmk}
Even without imposing the Galois condition or restricting the ramification points, \cref{defrthm} and \cref{tog1} show that $X$ has finitely many equivalence classes of maps of fixed degree to curves which are not $\mathbb{P}^1_k$. Hence, the new content of \cref{maingalthm} is resolving the case where the target curve is $\PP^1_k$.
\end{rmk}

As a corollary of \cref{maingalthm}, we have a statement about tame Galois branched covers similar to the case when considering branched covers of a fixed base:

\begin{cor}\label{notamegalcor}
Let $X$ be a smooth proper curve over an algebraically closed field $k$ and $G$ a finite group. There are only finitely many equivalence classes of $G$-Galois, tame covers with source $X$ with fixed ramification points. This implies that there are no non-constant tame $G$-Galois families.
\end{cor}

This follows since the conditions of \cref{maingalthm} require $G$ to have an element of order $p$ and be acting on $\PP^1_k$. By \cref{p1fp} below, every automorphism of $\PP^1_k$ has at least one stabilized point. Hence, such a cover is wildly ramified.

As mentioned in the introduction, \cref{maingalthm} will be proven in three cases: when the genus of $X$ is zero, one, or greater than one. The genus greater than one case is resolved by the following:

\begin{rmk}
Let $X$ be a smooth proper curve of genus greater than one. Then the automorphism group of $X$ is finite and so there are only finitely many subgroups of $\Aut(X)$. Therefore, by \cref{subgroupstocovers}, there are only finitely many equivalence classes of Galois branched covers with source $X$, even without specifying the ramification points or the group.
\end{rmk}

\subsection{Genus 1 case}\label{genusonegalsect}
For the entirety of this section, fix a smooth proper genus 1 curve $X$ over an algebraically closed field $k$ and a finite group $G$. Since $X$ is defined over an algebraically closed field, we may fix a $k$-rational base point $0$ to consider $X$ as an elliptic curve. Recall that $\Aut(X)\cong T_X \rtimes \Aut_0(X)$, where $T_X$ consists of the translations of $X$ by closed points and $\Aut_0(X)$ is the finite group of automorphisms of $X$ fixing the base point 0. We will denote the identity element of $\Aut_0(X)$ by 1.

First, we specify using the above isomorphism which automorphisms have fixed points and which are fixed point free:

\begin{lem}\label{fpfpf}
Fix $\varphi\in\Aut(X)$. Using the isomorphism $\Aut(X)\cong T_X \rtimes \Aut_0(X)$, write $\varphi=(P,\sigma)$ where $P$ is a closed point of $X$ and $\sigma\in\Aut_0(X)$. Then $\varphi$ is fixed point free precisely when $\sigma = 1$  and $P\neq 0$.
\end{lem}
\begin{proof}
Let $\varphi=(P,\sigma)$ be an arbitrary automorphism of $X$. Since $\varphi(Q)=\sigma(Q)+P$, the fixed points of $\varphi$ are the points $Q$ such that $(1-\sigma)(Q)=P$.

If $\sigma\neq 1$, then $(1-\sigma)$ is a surjective group homomorphism with kernel $K$ of cardinality equal to the degree of the map $(1-\sigma)$ ($\varphi$ being an automorphism implies that $(1-\sigma)$ is separable). Therefore we have a short exact sequence:

$$\xymatrix{0\ar[r] & K\ar[r] & X\ar[r]^{(1-\sigma)} & X\ar[r] & 0.}$$
The $|K|$ preimages of $P$ are precisely the fixed points of $\varphi$. Therefore, $\sigma\neq 1$ implies that $\varphi$ is not fixed point free.

Now suppose $\sigma=1$. Since the fixed points of $\varphi$ are the points $Q$ such that $(1-\sigma)(Q)=P$, $\varphi$ has fixed points if and only if $P=0$, in which case $\varphi$ is the identity.
\end{proof}

We now show that there are finitely many stabilized point free $G$-actions on $X$:

\begin{lem}\label{fpfellaction}
Let $G$ be a finite group. There are only finitely many equivalence classes of unramified $G$-Galois covers with source $X$. Stated another way, there are only finitely many $G$-actions on $X$ which are stabilized point free.
\end{lem}
\begin{proof}
Let $H\subseteq\Aut(X)$ be a stabilized point free $G$-action. By \cref{fpfpf}, each non-identity element of $H$ must be of the form $(P,1)$ where $P\neq 0$. Let $n=|G|$. Since $(P,1)\in H$, it must be the case that $P$ is an $n$-torsion point of $X$. There are only finitely many such $P$, and so there are finitely many possible such $H$.
\end{proof}

Next, we analyze automorphisms of $X$ with fixed points contained in a specified set:

\begin{lem}\label{fpellaction}
Let $S$ be a finite nonempty set of closed points of $X$. There are only finitely many $\varphi\in \Aut(X)$ which have non-trivial fixed locus contained in $S$.
\end{lem}
\begin{proof}
Fix $Q\in S$. Since any $\varphi$ with non-trivial fixed locus contained in $S$ must fix at least one point in $S$, it suffices to show that there are only finitely many automorphisms of $X$ which fix $Q$. Suppose $\varphi\in \Aut(X)$ is one such automorphism, and that it is not the identity. By \cref{fpfpf}, $\varphi$ is of the form $(P,\sigma)$ where $\sigma\neq 1$.  Since $(1-\sigma)$ is a surjective group homomorphism, we have a short exact sequence:

$$\xymatrix{0\ar[r] & K\ar[r] & X\ar[r]^{(1-\sigma)} & X\ar[r] & 0}$$
where $K$ is the kernel of $(1-\sigma)$. As in the proof of \cref{fpfpf}, the fixed points of $\varphi$ are precisely the points in the preimage of $P$ under $(1-\sigma)$. Therefore $P$ is the unique point for which $Q$ is a fixed point of an automorphism of the form $(R,\sigma)$. Since there are finitely many $\sigma\in\Aut_0(X)$, $Q$ is a fixed point for only finitely many $\varphi\in \Aut(X)$.
\end{proof}

Using \cref{fpellaction}, we can prove a strengthening of \cref{maingalthm} for ramified covers in the genus 1 case:

\begin{lem}\label{fpellthm}
Let $S$ be a finite set of closed points of $X$. There exist only finitely many group actions on $X$ with nonempty stabilized locus contained in $S$.
\end{lem}
\begin{proof}
Fix one such group action. Since the action has nonempty stabilized locus, let $\varphi$ be a non-identity element of the group action which has a fixed point. By \cref{fpellaction}, there are only finitely many automorphisms with non-trivial stabilized locus contained in $S$.  It therefore suffices to show that only finitely many fixed point free automorphisms of $X$ can occur in the same group action as $\varphi$.

Since $\varphi$ is not the identity and has fixed points, by \cref{fpfpf} it is of the form $(P,\sigma)$ where $\sigma\neq 1$. Let $\psi$ be a fixed point free automorphism of $X$. Again by \cref{fpfpf}, $\psi$ is of the form $(Q,1)$ where $Q\neq 0$. If $\varphi$ and $\psi$ occur in the same group action, so does $\psi\circ\varphi=(P+Q,\sigma)$. Since $\sigma\neq 1$, $(1-\sigma)$ is a surjective group homomorphism  and so it induces a short exact sequence:

$$\xymatrix{0\ar[r] & K\ar[r] & X\ar[r]^{(1-\sigma)} & X\ar[r] & 0}$$
where $K$ is the kernel of $(1-\sigma)$. The fixed points of $\psi\circ\varphi=(P+Q,\sigma)$ are the preimages of $P+Q$ under $(1-\sigma)$. However, $S$ is finite so there are only finitely many $Q$ for which the preimage of $P+Q$ is contained in $S$. This implies that only finitely many $\psi$ can occur in the same group action as $\varphi$ if the stabilized locus is contained in $S$.
\end{proof}

\begin{rmk}
Since $X$ is an elliptic curve, the induced cover being ramified is equivalent by the Riemann-Hurwitz formula to the target space being the projective line. Hence, another way to state \cref{fpellthm} is that there are only finitely many equivalence classes of Galois branched covers from a fixed elliptic curve $X$ to $\PP^1_k$ after fixing the ramification locus.
\end{rmk}

The genus 1 case of \cref{maingalthm} follows as a corollary of \cref{fpellthm} and \cref{fpfellaction}:
\begin{cor}
Let $X$ be a smooth proper genus 1 curve over an algebraically closed field $k$, $S$ a finite set of closed points of $X$, and $G$ a finite group. There exist only finitely many equivalence classes of $G$-Galois branched covers with ramification locus contained in $S$.
\end{cor}

\subsection{Genus 0 case}\label{genuszerogalsect}
For the entirety of this section we will assume $X=\mathbb{P}^1_k$, where $k$ is an algebraically closed field of characteristic $p$. When $p=0$, the statements concerning the case where $p$ is the order of an element or divides the order of a group are vacuous.

Fix a finite group $G$ and a finite set $S$ of closed points of $\PP^1_k$. To recall, we shall study the subgroups of $\Aut(\PP^1_k)$ isomorphic to $G$ with stabilized locus $S$. In doing so, we will make use of the following observation:

\begin{rmk}
Let $k$ be an algebraically closed field. Then $\Aut_k(\mathbb{P}^1)\cong \PGL_2(k)$. We shall fix the isomorphism between these two groups sending a fractional linear transformation $\varphi=(ax+b)/(cx+d)$ to the equivalence class of $$M(\varphi)=\begin{pmatrix} a & b \\ c & d\end{pmatrix}.$$ Under this isomorphism, $\varphi$ acting on a point with projective coordinates $[x:y]$ corresponds to $M(\varphi)$ acting on the column vector $[x\ y]$.
\end{rmk}

Throughout this section, we will use the following fact:

\begin{lem}\label{p1fp}
Let $\varphi\in \Aut(\PP_k^1)$ have finite order and not be the identity automorphism. Then $\varphi$ has exactly one or two fixed points, and has exactly one fixed point when $\varphi$ has order $p$.
\end{lem}
\begin{proof}
Using the isomorphism with $\PGL_2(k)$, the fixed points of $\varphi$ correspond to eigenvectors of $M(\varphi)$ up to scaling. Since $k$ is algebraically closed, $M(\varphi)$ is conjugate to a matrix of one of the following forms:

$$\begin{pmatrix} \lambda_1 & 0 \\ 0 & \lambda_2 \end{pmatrix},\ \begin{pmatrix} \lambda & 1 \\ 0 & \lambda \end{pmatrix}.
$$
The automorphisms of $\PP^1_k$ corresponding to conjugate matrices have the same number of fixed points, so it suffices to consider matrices of the above form. The first matrix has $[1\ 0]$ and $[0\ 1]$ as eigenvectors, and hence 0 and $\infty$ as fixed points. For the second matrix, up to scaling $[0\ 1]$ is the only eigenvector, and hence the only fixed point is $\infty$. One can easily check since $\cha(k)=p$ that the order of the first matrix in $\PGL_2(k)$ is prime to $p$ while the order of the second matrix is $p$ (or has infinite order if $p=0$).\end{proof}

We will divide the proof of \cref{maingalthm} into two cases: when $p$ does not divide $|G|$ and when it does. In \cite{fab}, Faber refers to the former as \textit{$p$-regular groups} and the latter as \textit{$p$-irregular groups}. The $p$-regular groups also occur as subgroups of $\PGL_2(\CC)$ (\cite[Theorem C]{fab})), but when $p>0$ there are $p$-irregular subgroups of $\PGL_2(k)$ which do not occur as subgroups of $\PGL_2(\CC)$ (\cite[Theorem 6.1]{fab}).

To simplify the statement of theorems and proofs for the remainder of this section, we introduce the following notation:

\begin{notation}
For a field $k$, $\mu_n(k)$ will denote the $n$th roots of unity in $k^\times$. A primitive $n$th root of unity in $k^\times$ will be denoted by $\zeta_n$.
\end{notation}

In the $p$-regular case we will make use of the following classification of subgroups up to conjugacy:

\begin{thm}(\cite[Theorem C]{fab})\label{tamep1ref}
Let $k$ be an algebraically closed field of characteristic $p$. Every finite $p$-regular subgroup of $\PGL_2(k)$ is conjugate to one of the following groups:

\benum
\item If $p\nmid n$, $\begin{pmatrix} \mu_n(k) & 0 \\ 0 & 1\end{pmatrix}$. This group is isomorphic to $\ZZ/n\ZZ$.

\item If $p\nmid n$ and $p\neq 2$, $\begin{pmatrix} \mu_n(k) & 0 \\ 0 & 1\end{pmatrix}\rtimes \left\langle\begin{pmatrix}0 & 1 \\ 1 & 0\end{pmatrix}\right\rangle$. This group is isomorphic to $D_{2n}$, the dihedral group of order $2n$.

\item If $p\neq 2,3$, $N\rtimes \left\langle \begin{pmatrix} 1 & \zeta_4 \\ 1 & -\zeta_4\end{pmatrix}\right\rangle$, where $N$ is generated by $\begin{pmatrix} 0 & -1 \\ 1 & 0\end{pmatrix}$ and $\begin{pmatrix} 0 & 1 \\ 1 & 0\end{pmatrix}$. This group is isomorphic to $A_4$.

\item If $p \neq 2,3$, the group generated by the group in (3) and $\begin{pmatrix} \zeta_4 & 0 \\ 0 & 1\end{pmatrix}$. This group is isomorphic to $S_4$.

\item If $p\neq 2,3,5$, the group generated by $\begin{pmatrix} \zeta_5 & 0 \\ 0 & 1\end{pmatrix}$ and $\begin{pmatrix} 1 & (1-\zeta_5-\zeta_5^{-1}) \\ 1 & -1\end{pmatrix}$. This group is isomorphic to $A_5$.
\eenum
\end{thm}

The main utility in the statement of \cref{tamep1ref} is that it classifies groups up to conjugacy rather than isomorphism. Since we are interested in specific occurrences of $G$ as a subgroup of $\PGL_2(k)$, this additional information greatly simplifies the proof of the following:

\begin{thm}\label{tamep1}
Let $k$ be an algebraically closed field of characteristic $p$, $S$ a finite set of closed points of $\PP^1_k$, and $G$ a finite $p$-regular group. There are only finitely many $G$-actions on $\PP^1_k$ with stabilized locus $S$.
\end{thm}
\begin{proof}
Let $G$ be a $p$-regular group. Since $G$ is $p$-regular, every subgroup corresponding to a $G$-action must be conjugate to one of the five groups listed in \cref{tamep1ref}. It therefore suffices to show that there are finitely many groups conjugate to each of the listed groups with the same stabilized locus. One can check directly that every group but the cyclic group in (1) has at least three stabilized points. Since an element in $\Aut(\PP^1_k)$ is determined by the image of three points, only finitely many elements can act as a permutation on $S$ when $|S|\geq 3$. This implies except in case (1) that there are finitely many conjugate groups with the same stabilized points.

We are left to consider the cyclic group in (1). The stabilized points of this group are 0 and $\infty$. The matrices which act as a permutation on the set $\{0,\infty\}$ are of the form $\begin{pmatrix} \alpha & 0 \\ 0 & 1\end{pmatrix}$ or $\begin{pmatrix} 0 & 1 \\ \beta & 0 \end{pmatrix}$ where $\alpha,\beta\in k^\times$. However, conjugating by a matrix of either form merely acts as an automorphism of the group, so $\begin{pmatrix} \mu_n(k) & 0 \\ 0 & 1\end{pmatrix}$ has no conjugate subgroups with the same stabilized points.
\end{proof}

In the $p$-irregular case we will make use of the following analog of \cref{tamep1ref}:

\begin{thm} (\cite[Theorem 6.1]{fab})\label{wildp1ref}
Let $k$ be an algebraically closed field of characteristic $p$. Every finite $p$-irregular subgroup of $\PGL_2(k)$ is conjugate to one of the following groups:

\benum
\item $\PSL_2(\FF_{p^n})$ for some $n\in\NN$.

\item $\PGL_2(\FF_{p^n})$ for some $n\in\NN$.

\item The group $\begin{pmatrix} 1 & \Gamma \\ 0 & 1\end{pmatrix}\rtimes\begin{pmatrix} \mu_n(k) & 0 \\ 0 & 1\end{pmatrix}$. Here $n\in\NN$ is prime to $p$ and $\Gamma$ is an additive subgroup of $k$ of rank $m\in\NN$ such that $\mu_n(k)\subseteq \Gamma$ and $\mu_n(k)\cdot \Gamma\subseteq \Gamma$. Two such subgroups of $\PGL_2(k)$ of the same order are conjugate if and only if $\Gamma'=\alpha\cdot\Gamma$ for some $\alpha\in k^\times$.

\item If $p=2$, the group $\begin{pmatrix} \mu_n(k) & 0 \\ 0 & 1 \end{pmatrix}\rtimes\left\langle\begin{pmatrix}0 & 1 \\ 1 & 0\end{pmatrix}\right\rangle$. Here $n\in\NN$ is odd and greater than one. This group is isomorphic to the dihedral group of order $2n$.

\item If $p=3$, the group generated by $\begin{pmatrix} \zeta_5 & 0 \\ 0 & 1\end{pmatrix}$ and $\begin{pmatrix} 1 & (1-\zeta_5-\zeta_5^{-1}) \\ 1 & -1\end{pmatrix}$. This group is isomorphic to $A_5$.
\eenum
\end{thm}

Again the classification of subgroups up to conjugacy greatly simplifies the proof of the following theorem:
\begin{thm}\label{wildp1}
Let $k$ be an algebraically closed field of characteristic $p$, $S$ a finite set of closed points of $\PP^1_k$, and $G$ a finite $p$-irregular group. There are infinitely many $G$-actions on $\PP^1_k$ with stabilized locus $S$ if and only if $G$ is isomorphic to $(\ZZ/p\ZZ)^m$. Moreover, each $(\ZZ/p\ZZ)^m$-action on $\PP^1_k$ has a unique stabilized point $P$. After fixing $P$, choosing such an action is equivalent to choosing a rank $m$ additive subgroup of $k$.
\end{thm}

\begin{proof}
Let $G$ be a $p$-irregular group. Since $G$ is $p$-irregular, every group corresponding to a $G$-action must be conjugate to one of the five groups listed in \cref{wildp1ref}. As in the proof of \cref{tamep1}, if $|S|\geq 3$, there will be only finitely many $G$-actions with stabilized locus $S$.

One can check immediately that the $A_5$ group from (5) and the dihedral group from (4) have at least three stabilized points.

To see that $\PSL_2(\FF_{p^n})$ and $\PGL_2(\FF_{p^n})$ have at least three stabilized points, note that irrespective of the field, both groups contain the matrices $\begin{pmatrix}1 & 0 \\1 & 1\end{pmatrix}, \begin{pmatrix}1 & 1 \\0 & 1\end{pmatrix}$, and $\begin{pmatrix}0 & 1 \\-1 & 0\end{pmatrix}.$
The first fixes 0, the second $\infty$, and the third the square roots of -1. This yields three stabilized points in characteristic 2 and four in higher characteristics.

Therefore we have only case (3) left to consider. First, consider the case where $n>1$. Then any subgroup of $\PGL_2(k)$ isomorphic to $G$ is conjugate to a group of the form $\begin{pmatrix} 1 & \Gamma \\ 0 & 1\end{pmatrix}\rtimes\begin{pmatrix} \mu_n(k) & 0 \\ 0 & 1\end{pmatrix}$ where $\mu_n(k)\subseteq \Gamma$ and $\mu_n(k)\cdot \Gamma\subseteq \Gamma$. An element of the form $\begin{pmatrix} \zeta_n & 0 \\ 0 & 1\end{pmatrix}$ has 0 and $\infty$ as fixed points. An element of the form $\begin{pmatrix} \zeta_n & g \\ 0 & 1\end{pmatrix}$ for $g$ nonzero has $\frac{g}{1-\zeta_n}\neq 0$ as a fixed point. Since there are at least three stabilized points, in this case there are finitely many $G$-actions with stabilized locus $S$ as well.

Finally, suppose $n=1$. Then $G$ is conjugate to a group of the form $\begin{pmatrix} 1 & \Gamma \\ 0 & 1\end{pmatrix}$ such that $\FF_p\subseteq \Gamma$ and $\Gamma$ is a rank $m$ additive subgroup of $k$. Note that in this case $G\cong (\ZZ/p\ZZ)^m$. The unique stabilized point of such a group is $\infty$. Since $\begin{pmatrix} \alpha & \beta \\ 0 & 1\end{pmatrix}$ where $\alpha,\beta\in k$ and $\alpha\neq 0$ are precisely the matrices which fix $\infty$, every subgroup of $\PGL_2(k)$ isomorphic to $G$ and with stabilized point $\infty$ is a conjugate of $\begin{pmatrix} 1 & \Gamma \\ 0 & 1\end{pmatrix}$ by such a matrix. These conjugates have the form $\begin{pmatrix} 1 & \alpha\cdot\Gamma \\ 0 & 1\end{pmatrix}$.

Conversely, let $\Gamma$ be any rank $m$ additive subgroup of $k$. One immediately sees that the only stabilized point of the group $\begin{pmatrix} 1 & \Gamma \\ 0 & 1\end{pmatrix}$ is $\infty$. Choose a nonzero element $\gamma\in\Gamma$. Then $\FF_p\subseteq \gamma^{-1}\cdot \Gamma$ and $\begin{pmatrix} 1 & \Gamma \\ 0 & 1\end{pmatrix}$ is conjugate to $\begin{pmatrix} 1 & \gamma^{-1}\cdot \Gamma \\ 0 & 1\end{pmatrix}$ via conjugation by $\begin{pmatrix} \gamma & 0 \\ 0 & 1\end{pmatrix}$. Therefore, the subgroups of $\PGL_2(k)$ isomorphic to $(\ZZ/p\ZZ)^m$ and with stabilized point $\infty$ are of the form $\begin{pmatrix} 1 & \Gamma \\ 0 & 1\end{pmatrix}$ where $\Gamma$ is any rank $m$ additive subgroup of $k$.
\end{proof}

By \cref{subgroupstocovers}, $G$-actions with stabilized locus $S$ correspond to equivalence classes of $G$-Galois branched covers with source $X$ and ramification locus $S$. Therefore, the final case of \cref{maingalthm} follows as a corollary of \cref{tamep1} and \cref{wildp1}:

\begin{cor}
Let $k$ be an algebraically closed field, $S$ a finite set of closed points of $\PP^1_k$, and $G$ a finite group. There are infinitely many equivalence classes of $G$-Galois branched covers with source $\PP^1_k$ and ramification locus $S$ if and only if all the following hold:

\benum
\item $\cha(k)=p>0$

\item $G\cong (\mathbb{Z}/p\ZZ)^m$, where $m\in\NN$

\item $|S|=1$.
\eenum

When these conditions hold, equivalence classes of such covers with ramification locus $S$ are in correspondence with rank $m$ additive subgroups of $k$.
\end{cor}

\end{document}